\documentclass{llncs}

\usepackage[bookmarks]{hyperref}
\usepackage{
	amscd,
	amsmath,
	amssymb,
	color,
	complexity,
	enumerate,
	graphicx,
	latexsym,
	mathrsfs,
	makecell,
	pgfgantt,
	rotating,
	thmtools,
	tikz,
	tikz-cd,
	url,
	verbatim,
	wrapfig
}
\usetikzlibrary{trees}
\usepackage[all]{xy}
\newcommand{\ra}{\rightarrow}
\newcommand{\Ra}{\Rightarrow}
\newcommand{\La}{\Leftarrow}
\newcommand{\eps}{\varepsilon}
\newcommand{\N}{\mathbb{N}}
\newcommand{\B}{\mathfrak{B}}
\newcommand{\Cant}{2^\omega}
\newcommand{\upto} {\upharpoonright}
\newcommand{\ML}{Martin-L\"of}
\newcommand{\MLR}{Martin-L\"of random}

\newcommand{\abs}[1]{\lvert#1\rvert}

\usepackage[utf8]{inputenc}
\title{KL-randomness and effective dimension under strong reducibility}
\titlerunning{Medvedev degrees and dimension}  
\author{Bj{\o}rn Kjos-Hanssen \inst{1}\orcidID{0000-0002-1825-0097}\thanks{This work was partially supported by a grant from the Simons Foundation (\#704836 to Bj\o rn Kjos-Hanssen).} \and
David J. Webb\inst{1}\orcidID{0000-0002-5031-7669}
}
\authorrunning{Main Author et al.} 
\tocauthor{Main Author, Author Two, Another Author}
\institute{
University of Hawai\textquoteleft i at M\=anoa, Honolulu HI 96822, U.S.A.,\\
\email{bjoern.kjos-hanssen@hawaii.edu, dwebb42@hawaii.edu},\\ WWW home page:
\texttt{http://math.hawaii.edu/wordpress/bjoern/}
}

\date{January 2021}

\begin{document}

\maketitle

\begin{abstract}
We show that the (truth-table) Medvedev degree KLR of Kolmogorov--Loveland randomness coincides with that of \ML\ randomness, MLR, answering a question of Miyabe.
Next, an analogue of complex packing dimension is studied which gives rise to a set of weak
truth-table Medvedev degrees isomorphic to the Turing degrees.
\end{abstract}

\section{Introduction}

	Computability theory is concerned with the relative computability of reals, and of collections of reals. The latter can be compared by various means, including Medvedev and Muchnik reducibility.
	Among the central collections considered are those of
	completions of Peano Arithmetic, Turing complete reals, Cohen generic reals, random reals, and various weakenings of randomness such as reals of positive effective Hausdorff dimension.

	Perhaps the most famous open problem in algorithmic randomness \cite{MR2732288,MR2548883} is whether Kolmogorov--Loveland randomness is equal to \ML\ randomness.
	Here we show that at least they are Medvedev equivalent.


	Randomness extraction in computability theory concerns whether reals that are close (in some metric) to randoms can compute random reals. A recent example is \cite{MR3721461}.
	That paper does for Hausdorff dimension what was done for a notion intermediate between packing dimension and Hausdorff dimension in \cite{MR3116543}.
	That intermediate notion, \emph{complex packing dimension}, has a natural dual which we introduce in this article.
	Whereas our result on KL-randomness is positive, we establish some negative (non-reduction) results for our new \emph{inescapable dimension} and for relativized complex packing dimension (in particular Theorem \ref{tt}). These results are summarized in Figure 1.
	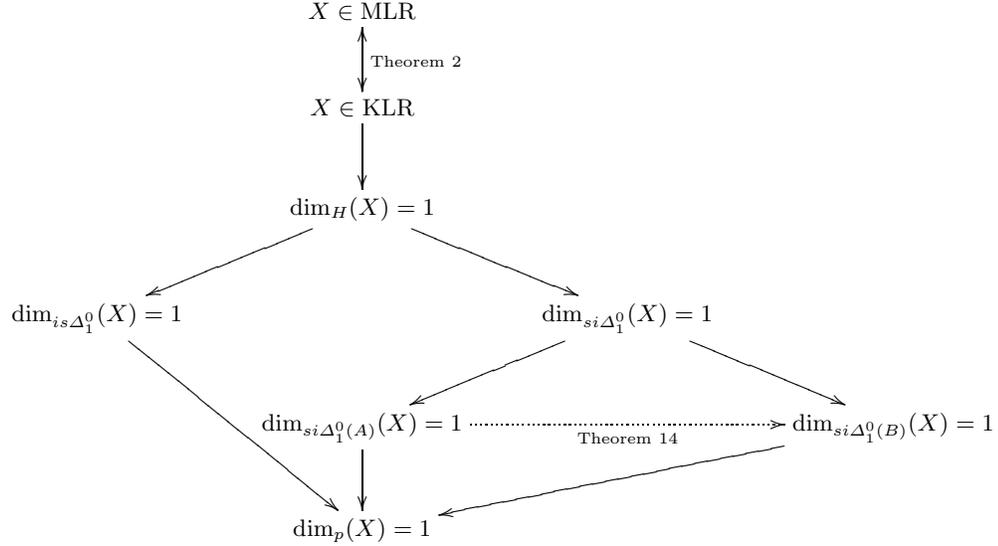
\begin{figure}
\[
\xymatrix{
&X\in\mathrm{MLR} \ar[d]\\
&X\in\mathrm{KLR}\ar[d]\ar_{\text{Theorem \ref{bjoern}}}[u]\\
&\dim_H(X)=1\ar[dl]\ar[dr] \\
\dim_{is\Delta^0_1}(X)=1\ar[ddr] && \dim_{si\Delta^0_1}(X)=1\ar[dl]\ar[dr]\\
&\dim_{si\Delta^0_1(A)}(X)=1\ar[d]\ar_{\text{Theorem \ref{tt}}}@{..>}[rr]&& \dim_{si\Delta^0_1(B)}(X)=1\ar[dll]\\ 
&\dim_p(X)=1
}
\]
    \caption{Truth-table Medvedev degrees of mass problems associated with randomness and dimension. Here $\mathcal C\to\mathcal D$ means $\mathcal C\ge_s \mathcal D$,
    dotted arrow means $\mathcal C\not\ge_s \mathcal D$ and we assume $A\not\le_TB$.}
    \label{fig:my_label}
\end{figure}

	Let CR, SR, KLR, and MLR be the classes of computably random, Schnorr random, Kolmogorov-Loveland random, and \MLR\ reals, respectively.
	For basic definitions from algorithmic randoness, the reader may consult two recent monographs \cite{MR2732288,MR2548883}.
	Let $\le_s$ denote the uniform (strong) reducibility of mass problems known as Medvedev reducibility, and let $\le_w$ denote the non-uniform (weak) version known as Muchnik reducibility.
	It was shown by Nies, Stephan and Terwijn \cite{MR2140044} that $\mathrm{CR}\le_w\mathrm{SR}$.
	Miyabe \cite{MR3890460} obtains an interesting counterpoint by showing as his main theorem that $\mathrm{CR}\not\le_s\mathrm{SR}$.

	\begin{theorem}[{\cite{MR2183813}}]\label{Merkle}
		Given a KL-random set $A=A_0\oplus A_1$, at least one of $A_0$, $A_1$ is ML-random.
	\end{theorem}

	As a corollary, $\mathrm{MLR}\le_w\mathrm{KLR}$. Miyabe \cite{MR3890460} posed one open problem ---
		is $\mathrm{MLR}\le_s\mathrm{KLR}$?
	--- which we answer in Theorem \ref{bjoern}.

	Let $K(\sigma)$ denote the prefix-free Kolmogorov complexity of a string $\sigma\in 2^{<\omega}$, and let $K_s(\sigma)$ be a computable nonincreasing approximation of $K(\sigma)$ in stages $s\in\omega$. The prefix of $A$ of length $n$ is denoted $A\upto n$.

	\begin{theorem}\label{bjoern}
		$\mathrm{MLR}\le_s\mathrm{KLR}$.
	\end{theorem}
	\begin{proof}
		Given a KL-random set $A=A_0\oplus A_1$, we output bits of either $A_0$
		or $A_1$, switching whenever we notice that the smallest possible
		randomness deficiency ($c$ such that $\forall n\,(K(A_i\upharpoonright
		n)\ge n-c)$) increases.
	
		This constant $c$ depends on $s$ and changes at stage $s+1$ if
		\[
			(\exists n\le s+1)\quad K_{s+1}(A_i\upto n)<n-c_s.
		\]
	
		By Theorem \ref{Merkle}, one of $A_0$, $A_1$ is ML-random, hence switching will occur only finitely often. Thus our output will have an infinite tail that is ML-random, and hence be itself ML-random.\qed 
	\end{proof}
	Inspection of the proof of Theorem \ref{bjoern} shows that we do not need the full power of Turing reductions, but have a truth-table reduction with use $\varphi(n)\le 2n$.

\section{Complex packing dimension and its analogue}\label{sec:intro}

	Let $K(\sigma)$ denote the prefix-free Kolmogorov complexity of a string $\sigma\in 2^{<\omega}$. The prefix of $A$ of length $n$ is denoted $A\upto n$.

	Viewed in terms of complexity \cite{MR2341912,MR1926330}, the Hausdorff and packing dimensions are dual to one another:
	
	\begin{definition}\label{Hausdorff}\label{packing}
		Let $A\in 2^{\omega}$. The effective Hausdorff dimension of $A$ is defined by
		\[
		\dim_{H}(A) = \sup_{m\in\N} \inf_{n\ge m} \dfrac{K(A\upto n)}{n}.
		\]
		The effective packing dimension of $A$ is
		\[
			\dim_{p}(A) = \inf_{m\in\N} \sup_{n\ge m} \dfrac{K(A\upto n)}{n}.
		\]
	\end{definition}
	
	Another notion of dimension was defined in previous work by Kjos-Hanssen and Freer \cite{MR3116543}, which we review here. Let $\mathfrak{D}$ denote the collection of all infinite $\Delta^0_1$ elements of $\Cant$. The complex packing dimension is defined as
	
	\begin{definition}\label{cp}
		$\displaystyle \dim_{cp}(A) = \sup_{N\in \mathfrak{D}} \inf_{n\in N} \dfrac{K(A\upto n)}{n}$.
	\end{definition}
	
	\noindent This leads naturally to a new notion, the dual of complex packing dimension:
	
	\begin{definition}\label{ines}
		$\displaystyle \dim_i(A) = \inf_{N\in\mathfrak{D}} \sup_{n\in N} \dfrac{K(A\upto n)}{n}.$
	\end{definition}
	
	This is the \emph{inescapable} dimension of $A$, so named because
	if $\dim_i(A) = \alpha$, every infinite computable collection of prefixes of $A$ must contain prefixes with relative complexity arbitrarily close to $\alpha$.
	For such a real, there is no (computable) escape from high complexity prefixes. 

	As Freer and Kjos-Hanssen show in \cite{MR3116543}, for any $A\in \Cant$, \[
	0\le \dim_H(A)\le \dim_{cp}(A)\le \dim_p(A)\le 1.\]
	The expected analogous result also holds:
	
	\begin{theorem}\label{iineq}
		For any $A\in \Cant$, $0\le \dim_H(A)\le \dim_{i}(A)\le \dim_p(A)\le 1$.
	\end{theorem}
	
	\begin{proof} As the sets $[n, \infty)$ are computable subsets of $\N$, $ \dim_{i}(A)\le \dim_p(A)$. For the second inequality, notice that for all $m\in\N$ and all $N\in\Delta^0_1$,
		\begin{eqnarray*}
			\inf_{n\in [m, \infty)} \dfrac{K(A\upto n)}{n}
			&\le& \inf_{n\in N\cap[m, \infty)}\dfrac{K(A\upto n)}{n}\\
			&\le& \sup_{n\in N\cap[m, \infty)}\dfrac{K(A\upto n)}{n}
			\le \sup_{n\in N}\dfrac{K(A\upto n)}{n}.\ \qed
		\end{eqnarray*}
	\end{proof}

	Unexpectedly, this is the best one can do.
	As we will see in the next section, while the Hausdorff dimension of a real is always lower than its packing dimension, any permutation is possible for the complex packing and inescapable dimensions of a real.

\section{Incomparability for inescapable dimension}
	We begin with a proof that the inescapable and complex packing dimensions are incomparable in the following sense: $\dim_{cp}(A)\le \dim_{cp}(B)$ does not imply $\dim_i(A)\le \dim_i (B)$, nor vice versa.
	In fact we show a stronger statement:

	\begin{theorem}\label{incomp}
	There exist $A$ and $B$ in $\Cant$ such that $\dim_{cp}(A)<\dim_{cp}(B)$, but $\dim_{i}(B)<\dim_i(A)$.
	\end{theorem}
  
	Recall that a real $A$ \emph{meets} a set of strings $S$ if there is some $\sigma\in S$ such that $\sigma$ is a prefix of $A$.
	Moreover, $A$ is weakly 2-generic if for each dense $\Sigma^0_2$ set of strings $S$, $A$ meets $S$ \cite{Generic}.

	For a real $A$, let us write $A[m,n]$ to denote the string $A(m) A(m+1) \dots A(n-1)$. For two functions $f(n), g(n)$ we write $f(n)\le^+ g(n)$ to denote $\exists c\forall n\, f(n)\le g(n)+c$.
	We write $f(n) = \mathcal{O}(g(n))$ to denote $\exists M\exists n_0\forall n>n_0\,f(n)\leq Mg(n)$.
	It will also be useful to have the following theorem of Schnorr at our disposal:
	
	\begin{theorem}\label{sch}
	$A$ is \MLR\ iff $n\le^+ K(A\upto n)$.
	\end{theorem}
	
	Finally, for a real $A$ and $n\in\omega$ we use the indicator function $1_A$ defined by \[1_A(n)=\begin{cases} 1& \text{if } n\in A,\\ 0&\text{otherwise.}\end{cases}\]

	\begin{proof}[of Theorem \ref{incomp}]
		Let $A$ be a weakly 2-generic real, and let $R$ be a \MLR\ real.
		Let $s_k=2^{k^2}$, $k_n = \max\{k\mid s_k\le n\}$, $C = (01)^\omega$. Define
		
		$$B(n) = R\left(n-s_{k_n}\right)\cdot 1_{C}(k_n).$$
		
		Unpacking this slightly, this is
		\[
			B(n) = \begin{cases}
					R\left(n - s_k\right) & \text{if }\ s_k\le n < s_{k+1} \text{ for some even }k,\\
					0 & \text{otherwise.}
					\end{cases}
		\]
		 In this proof, let us say that an $R$-segment is a string of the form \linebreak $B\upto [s_{2m},s_{2m+1})$ for some $m$, and say that a 0-segment is a string of the form \linebreak $B\upto [s_{2m+1},s_{2m+2})$ for some $m$. These are named so that a 0-segment consists of zeros, and an $R$-segment consists of random bits.
		Notice that by construction, each such segment is much longer than the combined length of all previous segments.
		This guarantees certain complexity bounds at the segments' right endpoints. For instance, $B$ has high complexity at the end of $R$-segments:
		for any even $k\in\N$,
		
		\begin{eqnarray*}
			s_{k+1} - s_k&\le^+&K\left(B\left[s_k, s_{k+1}\right]\right)\\
			&\le^+& K(B\upto s_k) + K(B\upto s_{k+1}) \le^+ 2 s_k + K(B\upto s_{k+1}).
		\end{eqnarray*}
		
		The first inequality holds by Theorem \ref{sch} because $B\left[s_k, s_{k+1}\right] = R\upto(s_{k+1} - s_k)$.
		The second (rather weak) inequality holds because from descriptions of $B\upto s_k$ and $B\upto s_{k+1}$ we can recover $B[s_k,s_{k+1}]$.
		Finally, $K(\sigma)\le^+ 2\abs{\sigma}$ is a property of prefix-free Kolmogorov complexity $K$. Combining and dividing by $s_{k+1}$ gives
		\begin{align}\label{using}
			s_{k+1} - 3s_k&\le^+K(B\upto s_{k+1})\nonumber \\
			1-3\cdot 2^{-(2k+1)}&\le\dfrac{K(B\upto s_{k+1})}{s_{k+1}}+ \mathcal{O}\left(2^{-(k+1)^2}\right) \quad\text{as $k\to\infty$.}
		\end{align}
		Dually, the right endpoints of $0$-segments have low complexity: for odd $k\in\N$,
		\begin{align*}
			K(B\upto s_{k+1})&\le^+ K(B\upto s_k) + K(B[s_k, s_{k+1}]) \le^+ 2s_k + 2\log(s_{k+1} - s_k).
		\end{align*}
		The first inequality is again the weak bound that $B\upto s_{k+1}$ can be recovered from descriptions of $B\upto s_k$ and $B[s_k, s_{k+1}]$.
		For the second, we apply the $2\abs{\sigma}$ prefix-free complexity bound to $B\upto s_k$, but also notice that since $B[s_k, s_{k+1}] = 0^{s_{k+1}-s_k}$, it can be recovered effectively from a code for its length.
		Combining and dividing by $s_{k+1}$, we have
		\begin{align}\label{using2}
			K(B\upto s_{k+1})&\le^+ 2s_k + 2(k+1)^2 \nonumber\\
			\dfrac{K(B\upto s_{k+1})}{s_{k+1}}&\le 2^{-(2k+1)} + \mathcal{O}\left(2^{-(k+1)^2}\right) \quad\text{as $k\to\infty$.}
		\end{align}
		Now we can examine the dimensions of $A$ and $B$. \\ \\
		\textbf{Claim 1:} $\dim_{cp}(B) = 1$.\\
		Let $R_n$ be the set of right endpoints of $R$-segments of $B$, except for the first $n$ of them --- that is, $R_n = \{s_{2k+1}\}_{k=n}^\infty$.
		Then the collection of these $R_n$ is a subfamily of $\mathfrak{D}$, so that a supremum over $\mathfrak{D}$ will be at least the supremum over this family.
		\begin{eqnarray*}
			\sup_{N\in\mathfrak{D}} \inf_{n\in N} \dfrac{K(B\upto n)}{n}&\ge& \sup_{n\in\N} \inf_{s\in R_n} \dfrac{K(B\upto s)}{s}\\
			\geq \sup_{n\in\N} \inf_{s\in R_n} 1-3\cdot 2^{-(2s+1)}
						 &=& \sup_{m\in\N} 1-3\cdot 2^{-(2m+1)}=1
		\end{eqnarray*}
		by (\ref{using}).

		\noindent \textbf{Claim 2:} $\dim_{i}(B) = 0$.\\
		Let $Z_n$ be the set of right endpoints of $0$-segments of $B$, except for the first $n$ of them: $Z_n = \{s_{2k}\}_{k=n}^\infty$. Similarly to Claim 1, we obtain
		\begin{eqnarray*}
			\inf_{N\in\mathfrak{D}} \sup_{n\in N} \dfrac{K(B\upto n)}{n}&\leq& \inf_{n\in\N} \sup_{s\in Z_n} \dfrac{K(B\upto s)}{s}\\
			\leq \inf_{n\in\N} \sup_{s\in Z_n} 2^{-(2s+1)}
						 &=& \inf_{m\in\N}2^{-(2m+1)}=0
		\end{eqnarray*}
		by (\ref{using2}).

		\noindent \textbf{Claim 3:} $\dim_{cp}(A) = 0$.\\
		For each natural $k$ and $N$ in $\mathfrak{D}$, the following sets are dense $\Sigma^0_1$:
		\[
			\left\{\sigma\in2^{<\omega} : \abs{\sigma}\in N \text{ and } (\exists s)\ K_s(\sigma)<\abs{\sigma}/k]\right\}.
		\]
		As $A$ is weakly 2-generic, it meets all of them. Hence
		\[
			\sup_{N\in\mathfrak{D}}\inf_{m\in N}\dfrac{K(\sigma\upto m)}{m}=0.
		\]
		\textbf{Claim 4:} $\dim_{i}(A) = 1$.\\
		For each natural $k$ and $N$ in $\mathfrak{D}$,
		\[
			\left\{\sigma\in2^{<\omega} : \abs{\sigma}\in N \text{ and } (\forall s)\ K_s(\sigma)>\abs{\sigma}(1-1/k)\right\}
		\]
		is a dense $\Sigma^0_2$ set. As $A$ is weakly 2-generic, it meets all of these sets. Hence \[
		\inf_{N\in\mathfrak{D}}\sup_{m\in N}\dfrac{K(A\upto m)}{m}=1.\ \qed
		\]
	\end{proof}

	We say that $A$ is finite-to-one reducible to $B$ if there is a total computable function $f:\omega\to\omega$ such that the preimage of each $n\in\omega$ is finite and for all $n$, $n\in A\iff f(n)\in B$.

	\begin{definition}
		Let $\B$ be a class of infinite sets downward closed under finite-to-one reducibility. For $A\in \Cant$, we define
		\[
			 \dim_{is\B}(A) = \inf_{N\in\B} \sup_{n\in N} \dfrac{K(A\upto n)}{n}\quad\text{and}\quad
			 \dim_{si\B}(A) = \sup_{N\in\B} \inf_{n\in N} \dfrac{K(A\upto n)}{n}.
		\]
	\end{definition}

	Notice that for any oracle $X$, the classes of infinite sets that are $\Delta^0_n(X), \Sigma^0_n(X)$ or $\Pi^0_n(X)$ are downward closed under finite-to-one reducibility, and so give rise to notions of dimension of this form.
	We will label these $\mathfrak{D}_n(X)$, $\mathfrak{S}_n(X)$, and $\mathfrak{P}_n(X)$ respectively, leaving off $X$ when $X$ is computable. Interestingly, for fixed $n$, the first two give the same notion of dimension.
	
	\begin{theorem} For all $A\in\Cant$ and $n\in\N$, $\dim_{is\Sigma^0_n}(A) = \dim_{is\Delta^0_n}(A)$.
	\end{theorem}
	
	\begin{proof}
		We prove the unrelativized version of the statement, $n=1$.\\
		\textbf{[$\le$]} As $\Delta^0_1\subseteq \Sigma^0_1$, this direction is trivial.\\
		\textbf{[$\ge$]} As every infinite $\Sigma^0_1$ set $N$ contains an infinite $\Delta^0_1$ set $N'$, we have
		\begin{align*}
			\dim_{is\Sigma^0_1}(A)  = \inf_{N\in \mathfrak{S}_1} \sup_{n\in N} \dfrac{K(A\upto n)}{n}
			&\ge \inf_{N\in \mathfrak{S}_1} \sup_{n\in N'} \dfrac{K(A\upto n)}{n}\\
			&\ge \inf_{N\in \mathfrak{D}_1} \sup_{n\in N} \dfrac{K(A\upto n)}{n}
			= \dim_{is\Delta^0_1}(A).\qed
		\end{align*}
	\end{proof}
	
By a similar analysis, the analogous result for $si$ dimensions is also true.
	\begin{theorem} For all $A\in\Cant$ and $n\in\N$, $\dim_{si\Sigma^0_n}(A) = \dim_{si\Delta^0_n}(A)$.\label{sigeqdelsi}\end{theorem}

	What about the $\Pi^0_n$ dimensions? Unlike the $\Sigma^0_1$ case, these do not collapse to any other dimension. Two lemmas will be useful in proving this.
	The first (which was implicit in Claims 1 and 2 of Theorem \ref{iineq}) will allow us to show that an $si$-dimension of a real is high by demonstrating a sequence that witnesses this.
	The second is a generalization of the segment technique, forcing a dimension to be $0$ by alternating $0$- and $R$-segments in a more intricate way, according to the prescriptions of a certain real.
	The constructions below proceed by selecting a real that will guarantee that one dimension is 0 while leaving room to find a witnessing sequence for another.

	\begin{lemma}[Sequence Lemma]\label{seqlem}
Let $\B$ be a class of infinite sets downward closed under finite-to-one reducibility, and let $N= \{n_k\mid k\in\omega\}\in\B$.
	\begin{enumerate}
		\item \label{one} If $\displaystyle  \lim_{k\ra\infty}\dfrac{K(X\upto n_k)}{n_k}=1$, then $\dim_{si\B}(X) = 1$.
		\item\label{two} If $\displaystyle\lim_{k\ra\infty}\dfrac{K(X\upto n_k)}{n_k}=0$, then $\dim_{is\B}(X)=0$.
	\end{enumerate}
	\end{lemma}
	\begin{proof}
		We prove (\ref{one}); (\ref{two}) is similar.

		Form the infinite family of sets $\{N_m\}$ defined by $N_m = \{n_k\mid k\geq m\}$. From the definition of the limit, for any $\eps>0$ there  is an $l$ such that $$\inf_{N_l}\dfrac{K(X\upto n_k)}{n_k}> 1-\eps.$$ As $\eps$ was arbitrary, $$\sup_m \inf_{N_m}\dfrac{K(X\upto n_m)}{n_m}= 1.$$
		Thus as $\B$ is closed under finite-to-one reduction, the $N_m$ form a subfamily of $\B$, so that $\sup_{N\in \mathfrak B} \inf_{n\in N} K(X\upto n)/n= 1$.\qed
	\end{proof}

	Recall that an infinite real $A$ is said to be \emph{immune} to a class $\B$ if there is no infinite member $B\in\B$ such that $B\subseteq A$ as sets, or \emph{co-immune} to a class $\B$ if its complement is immune to $\B$. We will sometimes refer to these properties as $\B$-immunity or $\B$-co-immunity, respectively.

	\begin{lemma}[Double Segment Lemma]\label{dslem} Let $X_0\in \Cant$ be such that $X_0$ is  $\B$-immune for a class $\B$  of infinite sets downward closed under finite-to-one reducibility. Set $X = X_0\oplus X_0$. Let $s_k=2^{k^2}$, and $k_n = \max\{\text{odd }k\mid s_k\le n\}$. Let $A$ be an arbitrary real and let $R$ be \MLR. \begin{enumerate}
		\item\label{dslem-one} If $B = A\left(n-s_{k_n}\right)\cdot 1_{X}(k_n)$, then $\dim_{si\B}(B) = 0$.
		\item\label{dslem-two} If $B = R\left(n-s_{k_n}\right)\cdot 1_{\overline{X}}(k_n)$, then $\dim_{is\B}(B) = 1$.
	\end{enumerate}
	\end{lemma}
		Again, we will give a detailed proof of only the $\dim_{si\B}$ result (though the necessary changes for $\dim_{is\B}$ are detailed below). Unpacking the definition of $B$,
		\[
			B(n) = \begin{cases}
					A\left(n - s_k\right) & \text{if $k_n\in X$}\\
					0 & \text{otherwise.}
					\end{cases}
		\] $B$ is here built out of segments of the form $B\left[s_{k_n}, s_{k_n+2}\right]$ for odd $k$. Here a segment is a $0$-segment if $k_n\not\in X$, or an $A$-segment if $k_n\in X$, which by definition is a prefix of $A$. These segments are now placed in a more intricate order according to $X$, with a value $n$ being contained in a $0$-segment if $X(k_n) = 0$, and in an $A$-segment if $X(k_n) = 1$. With some care, this will allow us to leverage the $\B$-immunity of $X_0$ to perform the desired complexity calculations.
			
        Specifically, we want to show that for any $N\in\B$, $\inf_NK(B\upto n)/n=0$. It is tempting to place the segments according to $X_0$ and invoke its $\B$-immunity to show that for any $N\in\B$, there are infinitely many $n\in N$ such that $n$ is in a $0$-segment, and argue that complexity will be low there. The problem is that we have no control over \emph{where} in the $0$-segment $n$ falls. Consider in this case the start of any segment following an $A$-segment: $n=s_{k_n}$ for $k_{n}-1\in X_0$ and $k_n\in X_0$. We can break $A$ and $B$ into sections to compute \begin{align*}K(A\upto n) &\leq^+ K(A\upto(n-s_{k_n-1}))+K(A[n-s_{k_n-1}, n]) \\
        &=K(B[s_{k_n-1},n])+K(A[n-s_{k_n-1}, n])\tag{$k_n-1\in X_0$}\\
        &\leq^+K(B\upto n) + K(B\upto s_{k_n-1})+K(A[n-s_{k_n-1}, n])\\
         K(A\upto n)&\leq^+K(B\upto n)+4s_{k_n-1}\tag{$K(\sigma)\leq^+2|\sigma|$}
        \end{align*} Even if $n$ is the start of a $0$-segment, if $K(A\upto n)$ is high, $K(B\upto n)$ may not be as low as needed for the proof. Our definition of $X$ avoids this problem:
    	\begin{proof}[of Theorem \ref{dslem}]
			Suppose for the sake of contradiction that for some $N\in\B$, there are only finitely many $n\in N$ with $k_n, k_n-1\in \overline{X}$, i.e., that are in a $0$-segment immediately following another $0$-segment.
			Removing these finitely many counterexamples we are left with a set $N'\in\B$ such that for all $n\in N'$, $\lnot[(k_n\not\in X) \land (k_n-1\not\in X)].$
			As $k_n$ is odd, the definition of $X$ gives that $\lfloor k_n/2\rfloor\in X_0$.
			By a finite-to-one reduction from $N'$, the infinite set $\{\lfloor k_n/2\rfloor\}_{n\in N'}$ is a member of $\B$ and is contained in $X_0$ - but $\overline{X_0}$ is immune to such sets.
			
			Instead it must be the case that there are infinitely many $n\in N$ in a $0$-segment following a $0$-segment, where the complexity is \begin{align*}
			K(B\upto n) &\leq^+ K\left(B\upto s_{n_{k-1}}\right)+ K\left(B\left[s_{n_{k-1}},n\right]\right)\\
			&\le^+ 2\cdot s_{n_{k-1}} + 2\log\left(n-s_{n_{k-1}}\right).
			\end{align*}
			Here the second inequality follows from the usual $2\abs{\sigma}$ bound and the fact that $B \left[s_{n_{k-1}},n\right]$ contains only $0$s. As $2^{k_n^2}\le n$, we can divide by $n$ to get
			\begin{align*}
			\dfrac{K(B\upto n)}{n}\le^+ \dfrac{2^{k_n^2-2k_n}}{2^{k_n^2}} + \dfrac{2\log(n)}{n}= 2^{-2k_n}+\dfrac{2\log(n)}{n}.
			\end{align*}
			As there are infinitely many of these $n$, it must be that $\inf_{n\in N} K(B\upto n)/n = 0$.
			This holds for every real $N$ with property $\B$, so 
			taking a supremum gives the result.
			
			The $\dim_{is\B}$ version concerns reals $B$ constructed in a slightly different way. Here, the same argument now shows there are infinitely many $n\in N$ in an $R$-segment following an $R$-segment. At these locations, the complexity $K(B\upto n)$ can be shown to be high enough that $\sup_N K(B\upto n)/n = 1$, as desired. 
			\qed
	\end{proof}

	With these lemmas in hand, we are ready to prove the following theorem:
	
	\begin{theorem}\label{pidelsepsi}
		For all natural $n$ there is a set $A$ with $\dim_{si\Pi^0_n}(A) = 1$ and $\dim_{si\Delta^0_n}(A)=0$.
	\end{theorem}
	
	\begin{proof}
		We prove the $n=1$ case, as the proofs for higher $n$ are analogous.
		
		Let $S_0$ be a co-c.e.\ immune set, and let $R$ be \MLR. Set $S = S_0\oplus S_0$, and define $k_n = \max\{\text{odd } k\mid 2^{k^2}\le n\}$. To build $A$ out of $0$-segments and $R$-segments, define $A(n) = R\left(n-2^{k_n^2}\right)\cdot 1_{S}(k_n)$.
		
		As $S_0$ is $\Pi^0_1$, so is $S$. Thus the set of right endpoints of $R$-segments,\linebreak $M = \left\{2^{k^2}\mid k\text{ is odd and }k-1\in S\right\}$ is also $\Pi^0_1$.
		By construction $\lim_{m\in M}K(A\upto m)/m = 1$ and thus the Sequence Lemma \ref{seqlem} gives that $\dim_{si\Pi^0_1}(A) = 1$.

		As the complement of a simple set is immune, the Double Segment Lemma \ref{dslem} shows that $\dim_{si\Delta^0_1}(A) = 0$.\qed
	\end{proof}

	The proof of analogous result for the $is$-dimensions is similar, using the same $S_0$ and $S$, and the real defined by $B(n) = R\left(n-2^{k_n^2}\right)\cdot 1_{\overline{S}}(k_n)$.
	\begin{theorem} For all $n\ge 1$ there exists a set $B$ with $\dim_{is\Pi^0_n}(B)=1$ and $\dim_{is\Delta^0_{n}}(B) = 0$.\end{theorem}

	It remains to show that the $\Delta^0_{n+1}$ and $\Pi^0_n$ dimensions are all distinct.
	We can use the above lemmas for this, so the only difficulty is finding sets of the appropriate arithmetic complexity with the relevant immunity properties.

	\begin{lemma}\label{piim} For all $n\geq 1$, there is an infinite $\Delta^0_{n+1}$ set $S$ that is $\Pi^0_n$-immune.
	\end{lemma}
	
    \begin{proof} We prove the unrelativized version, $n=1$. Let $C$ be a $\Delta^0_2$ cohesive set that is not co-c.e, i.e., for all $e$ either $W_e\cap C$ or $\overline{W_e}\cap C$ is finite. As $\overline{C}$ is not c.e.\ it cannot finitely differ from any $W_e$, so for all $e$, $W_e\setminus \overline{C} = W_e\cap C$ is infinite. Hence if $\overline{W_e}\subseteq C$, then by cohesiveness, $\overline{W_e}\cap C = \overline{W_e}$ is finite.\qed
\end{proof}

	\begin{theorem} For all $n\ge 1$ there exists a set $A$ with $\dim_{si\Delta^0_{n+1}}(A) = 1$ and $\dim_{si\Pi^0_n}(A)=0$.\label{pidelsepsin+1}\end{theorem}

	\begin{proof}
		This is exactly like the proof of Theorem \ref{pidelsepsi}, but $S_0$ is now the $\Pi^0_1$-immune set guaranteed by Lemma \ref{piim}.\qed
	\end{proof}

	Again, the analogous result for $is$-dimensions is similar:
	
	\begin{theorem} For all $n\geq 1$ there exists a set $B$ with $\dim_{is\Pi^0_n}(B)=1$ and $\dim_{is\Delta^0_{n+1}}(B) = 0$.\label{pidelsepisn+1}\end{theorem}
	
	After asking questions about the arithmetic hierarchy, it is natural to turn our attention to the Turing degrees. As the familiar notion of $B$-immunity for an oracle is exactly $\Delta^0_1(B)$-immunity for a class, we have access to the usual lemmas.
	We shall embed the Turing degrees into the $si$-$\Delta^0_1(A)$ dimensions (and dually, $is$-$\Delta^0_1(A)$). First, a helpful lemma:
	
	\begin{lemma}[Immunity Lemma]\label{immlem} If $A\nleq_T B$, there is an $S\le_T A$ such that $S$ is $B$-immune.\end{lemma}
	
	\begin{proof}
		Let $S$ be the set of finite prefixes of $A$. If $S$ contains a $B$-computable infinite subset $C$, then we can recover $A$ from $C$, but then $A\leq_T C\leq_T B$.\qed
	\end{proof}
	
	\begin{theorem}[{$si$-$\Delta^0_1$ Embedding Theorem}]\label{delembedding} Let $A,B\in \Cant$. Then  $A\le_T B$ iff
	for all $X\in \Cant, \dim_{si\Delta^0_1 (A)} (X) \le \dim_{si\Delta^0_1 (B)} (X)$.
	\end{theorem}
	\begin{proof}
		\textbf{[$\Ra$]} Immediate, as $\Delta^0_1(A)\subseteq \Delta^0_1(B)$.\\
		\textbf{[$\La$]} This is again exactly like the proof of Theorem \ref{pidelsepsi}, now using the set guaranteed by the Immunity Lemma \ref{immlem} as $S_0$.\qed
	\end{proof}

The result for $is$-dimensions is again similar:

\begin{theorem}[{$si$-$\Delta^0_1$ Embedding Theorem}] Let $A,B\in \Cant$. Then  $A\le_T B$ iff
	for all $X\in \Cant, \dim_{is\Delta^0_1 (A)} (X) \ge \dim_{is\Delta^0_1 (B)} (X)$.
	\end{theorem}
	
	We can push this a little further by considering weak truth table reductions: 

\begin{definition} $A$ is \emph{weak truth table reducible to} $B$ ($A\leq_{wtt} B$) if there exists a computable function $f$ and an oracle machine $\Phi$ such that $\Phi^B = A$, and the use of $\Phi^X(n)$ is bounded by $f(n)$ for all $n$ ($\Phi^X(n)$ is not guaranteed to halt).
\end{definition}
	\begin{theorem}\label{tt}
		If $A\not\le_T B$, then for all wtt-reductions $\Phi$ there exists an $X$ such that $\dim_{si\Delta^0_1(A)}(X)=1$ and, either $\Phi^X$ is not total or $\dim_{si\Delta^0_1(B)}(\Phi^X)=0$.
	\end{theorem}
	\begin{proof}
		Let $A\not\le_T B$, and let $\Phi$ be a wtt-reduction. Let $f$ be a computable bound on the use of $\Phi$, and define $g(n) = \max\{f(i)\mid i\le n\}$, so that $K(\Phi^X\upto n)\le^+ K(X\upto g(n)) + 2\log(n)$. For notational clarity, for the rest of this proof we will denote inequalities that hold up to logarithmic (in $n$) terms as $\leq^{\log}$.
		
		Next, we define two sequences $\ell_k$ and $\lambda_k$ which play the role $2^{k^2}$ played in previous constructions:
		\begin{align*}
			\ell_0 = \lambda_0 = 1, &&\lambda_k = \lambda_{k-1} +  \ell_{k-1}, && \ell_k = \min\left\{2^{n^2}\mid g(\lambda_{k})<2^{n^2}\right\}.
		\end{align*}
	    These definitions have the useful consequence that $\lim_k \ell_{k-1}/\ell_{k} = 0$. To see this, suppose $\ell_{k-1} = 2^{(n-1)^2}$. As $g$ is an increasing function, the definitions give
		$$\ell_k>g(\lambda_{k}) \ge \lambda_{k}=\lambda_{k-1} + \ell_{k-1}\ge \ell_{k-1}=2^{(n-1)^2}.$$
		Hence $\ell_{k}\ge 2^{n^2}$, so that $\ell_{k-1}/\ell_k \le 2^{-2n+1}$. As $\ell_k>\ell_{k-1}$ for all $k$, this ratio can be made arbitrarily small, giving the limit.

        A triple recursive join operation is defined by
        \[
        \bigoplus_{i=0}^2 A_i = \{3k+j \mid k\in A_j,\quad 0\le j\le 2\},\quad A_0,A_1,A_2\subseteq\omega.
        \]
		
		Let $S_0\leq_T A$ be as guaranteed by Lemma \ref{immlem}, and define $S=\bigoplus_{i=0}^2 S_0$.
		Let $R$ be \MLR, and define $X(n) = R\left(n - \ell_{k_n}\right) \cdot 1_{S}(k_n)$, where $k_n = \max\{k = 2 \pmod 3\mid \ell_k\le n\}$. This definition takes an unusual form compared to the previous ones we have seen in order to handle the interplay between $\lambda_k$ and $\ell_k$ - specifically the growth rate of $g(n)$. We are effectively ``tripling up" bits of $S_0$ (rather than doubling them as before) to account for the possibility that $g(n)$ grows superexponentially, with the condition that $k=2\pmod 3$ replacing the condition that $k$ is odd.\\
	
		\noindent \textbf{Claim 1:} $\dim_{si\Delta^0_1(A)}(X) = 1$.\\
		\emph{Proof:} As $N= \left\{\ell_k\right\}_{k\in S}$ is an $A$-computable set, by the Sequence Lemma \ref{seqlem} it suffices to show that
		\(
			\lim_{k\in S} K(X\upto \ell_k)/\ell_k = 1.
		\)
		For $\ell_k\in N$,
		\begin{align*}
			K(X\upto \ell_k)&\ge^{\log} K(X[\ell_{k-1}, \ell_{k}]) - K(X\upto \ell_{k-1})\\
			&\ge\ \ K(R\upto(\ell_k-\ell_{k-1})) - 2\ell_{k-1} \tag{as $k\in S$}\\
			&\ge^{\log} \ell_k - \ell_{k-1} - 2\ell_{k-1} \tag{as $R$ is \MLR}\\
			\dfrac{K(X\upto \ell_k)}{\ell_k}&\ge^{\log} \dfrac{\ell_k - 3\ell_{k-1}}{\ell_k}= 1 - 3\dfrac{\ell_{k-1}}{\ell_k}.
		\end{align*}
		which gives the desired limit by the above.
		
		\noindent\textbf{Claim 2:} $\dim_{si\Delta^0_1(B)}(\Phi^X) = 0$.\\
		\emph{Proof:} Suppose $N\leq_T B$. For notation, define $a = k_{g(n)}$. By mimicking the proof of Lemma \ref{dslem}, we can use the $B$-immunity of $S$ to show that there are infinitely many $n\in N$ such that $g(n)$ is in a $0$-segment following two $0$-segments, i.e., $a-2, a-1, a\not\in S$. By the definition of $X$,
		$$X[\ell_{a-2}, \ell_{a+1}] = 0^{\ell_{a+1} - \ell_{a-2}}.$$
		Suppose the value $X(m)$ is queried in the course of computing $\Phi^X\upto n$. By the definitions of $g$, $a$, and $\ell_k$,
		$m\leq g(n)<\ell_{a+1}.$
		Hence either $m<\ell_{a-2}$ or $m\in[\ell_{a-2}, \ell_{a+1}]$, so that $X(m)=0$. Thus to compute $\Phi^X\upto n$, up to a constant it suffices to know $X\upto \ell_{a-2}$. Thus
		$$K(\Phi^X\upto n)\le^+ K(X\upto \ell_{a-2})\leq^+ 2\ell_{a-2}$$
        As $g(n)>\ell_{a}$ it must be that $n>\lambda_a$. Dividing by $n$, we find that
		\begin{align*}
		\dfrac{K(\Phi^X\upto n)}{n}\leq^+\dfrac{2\ell_{a-2}}{\lambda_a}<
		 \dfrac{2\ell_{a-2}}{\lambda_{a-1}+\ell_{a-1}}
		 <\dfrac{2\ell_{a-2}}{\ell_{a-1}}.
		 \end{align*}
		As there are infinitely many of these $n$, it must be that $\inf_{n\in N} K(\Phi^X\upto n)/n = 0$.
			This holds for every $N\leq_T B$, so taking a supremum gives the result.\qed\end{proof}
\emph{Remark.}  We only consider $si$-dimensions for this theorem, as it is not clear what an appropriate analogue for $is$-dimensions would be. The natural dual statement for $is$-dimensions would be that for all reductions $\Phi$ there is an $X$ such that $\dim_{is\Delta^0_1(A)}(X)=0$, and either $\Phi^X$ is not total or $\dim_{is\Delta^0_1(B)}(\Phi^X)=1$.
But many reductions use only computably much of their oracle, so that $\Phi^X$ is a computable set.
This degenerate case is not a problem for the $si$ theorem, as its conclusion requires $\dim_{\Delta^0_1(B)}(\Phi^X)=0$. But for an $is$ version, it is not even enough to require that $\Phi^X$ is not computable - consider the reduction that repeats the $n$th bit of $X$ $2n-1$ times, so that $n$ bits of $X$ suffice to compute $n^2$ bits of $\Phi^X$.
Certainly $\Phi^X\equiv_{wtt} X$, so that $\Phi^X$ is non-computable iff $X$ is.
But 
\begin{align*} \dfrac{K(\Phi^X\upto n)}{n}\leq^+\dfrac{K(X\upto \sqrt{n})}{n}\leq^+\dfrac{2\sqrt{n}}{n}
\end{align*}
for all $n$, so that $\dim_p(\Phi^X) = 0$, and hence all other dimensions are 0 as well.

	\bibliographystyle{plain}
	\bibliography{cie21webb}
\end{document}